\documentclass[12pt,reqno]{amsart}

\usepackage{latexsym}
\usepackage{amssymb}
\usepackage{amsthm}
\usepackage{mathtools}
\usepackage{xfrac}
\usepackage{hyperref}
\usepackage{tikz-cd}
\usetikzlibrary{decorations.pathmorphing}
\usepackage{color}
\usepackage{enumitem}
\usepackage{rotating}

\usepackage{amsmath,amsfonts,amssymb}
\DeclareMathOperator\image{im}
\DeclareMathOperator\real{Re}

\DeclareMathOperator\id{Id}

\DeclareMathOperator{\interior}{int}

\newcommand{\map}[3]{\ensuremath{#1\colon#2\rightarrow#3}}
\newcommand{\dimNG}[1]{\ensuremath{\dim_{\mathcal{N}(#1)}}}
\newcommand{\sutdec}[1]{\overset{#1}{\rightsquigarrow}}

\newcommand{\N}{\ensuremath{\mathbb{N}}}  
\newcommand{\Z}{\ensuremath{\mathbb{Z}}}   
   
\newcommand{\R}{\ensuremath{\mathbb{R}}} 
\newcommand{\C}{\ensuremath{\mathbb{C}}}

\newcommand{\Poincare}{Poincar\'e\ }

\newcommand{\HH}{\ensuremath{H^{(2)}}}
\newcommand{\BB}{\ensuremath{b^{(2)}}}
\newcommand{\CC}{\ensuremath{C^{(2)}}}

\newcommand\isofrom{\xleftarrow{
		\,\smash{\raisebox{-0.65ex}{\ensuremath{\scriptstyle\sim}}}\,}}

\newcounter{restatesection}
\stepcounter{restatesection}

\newtheorem{lemma}{Lemma}[section]
\newtheorem{cor}[lemma]{Corollary}
\newtheorem{ques}[lemma]{Question}
\newtheorem{theorem}[lemma]{Theorem}
\newtheorem{mytheorem}{Theorem}[restatesection]
\newtheorem{prop}[lemma]{Proposition}

\theoremstyle{definition}
\newtheorem{defn}[lemma]{Definition}

\newtheorem*{conv*}{Convention}
\newtheorem*{lemma*}{Lemma}
\newtheorem*{exercise*}{Exercise}

\theoremstyle{remark}
\newtheorem{rem}[lemma]{Remark}
\newtheorem{examp}[lemma]{Example}
\newtheorem*{claim}{Claim}

\begin{document}
\author{Gerrit Herrmann}
\address {Fakult\"at f\"ur Mathematik \\
		Universit\"at Regensburg\\
		Germany\\
	\newline\url{www.gerrit-herrmann.de} }
\email{gerrit.herrmann@mathematik.uni-regensburg.de}	
\title{Sutured manifolds and $\ell^2$-Betti numbers}
\begin{abstract}
Using the virtual fibering theorem of Agol we show that a sutured 3-manifold $(M, R_+,R_-,\gamma)$ is taut if and only if the $\ell^2$-Betti numbers of the pair $(M,R_-)$ are zero. As an application we can characterize Thurston norm minimizing surfaces in a 3-manifold $N$ with empty or toroidal boundary by the vanishing of certain $\ell^2$-Betti numbers.     
\end{abstract}
\maketitle
\section{Introduction}
%
%
%

	
%

A sutured manifold $(M,R_+,R_-,\gamma)$ is a compact, oriented 3-manifold $M$ together with a set of disjoint oriented annuli and tori $\gamma$ on $\partial M$ which decomposes $\partial M\setminus\mathring{\gamma}$ into two subsurfaces $R_-$ and $R_+$. We refer to Section \ref{sec:Definitions} for a precise definition.
We say that a sutured manifold is \emph{balanced} if $\chi(R_+)=\chi(R_-)$. Balanced sutured manifolds arise in many different contexts. For example 3-manifolds cut along non-separating surfaces can give rise to balanced sutured manifolds.

Given a surface $S$ with connected components $S_1\cup\ldots\cup S_k$ we define its complexity to be $\chi_-(S)=\sum_{i=1}^{k}\max \left\{-\chi(S_i),0 \right\}$. A sutured manifold is called \emph{taut} if $R_+$ and $R_-$ have the minimal complexity among all surfaces representing $[R_-]=[R_+]\in H_2(M,\gamma;\Z)$.
\begin{theorem}[Main theorem]\label{thm:main}
Let $(M,R_+,R_-,\gamma)$ be a connected irreducible balanced sutured manifold. Assume that each component of $\gamma$ and $R_\pm$ is incompressible and $\pi_1(M)$ is infinite, then the following are equivalent
\begin{enumerate}[font=\normalfont]
	\item the manifold $(M,R_+,R_-,\gamma)$ is taut,\label{enm:taut}
	\item the $\ell^2$-Betti numbers of $(M,R_-)$ are all zero i.e.\ $\HH_\ast (M,R_-)=0$\label{enm:acyclic},
	\item the map $\HH_1(R_-)\to\HH_1(M)$ is a weak isomorphism.
\end{enumerate}
\end{theorem}
\begin{rem}
The same statement holds true if one replaces $R_-$ with $R_+$.
\end{rem}
By Gabai's theory of sutured manifold decompositions we obtain the following result about Thurston norm minimizing surfaces (see Section~\ref{sec:Definitions} for a definition) in an irreducible $3$-manifold $N$ with empty or toroidal boundary.

\begin{theorem}\label{thm:main2}
	Let $N$ be a connected orientable irreducible compact 3-manifold with empty or toroidal boundary and $\Sigma\hookrightarrow N$ a properly embedded decomposition surface, then the following are equivalent
\begin{enumerate}[font=\normalfont]
	\item $\Sigma$ is Thurston norm minimizing in the sense of Section~\ref{sec:Definitions},
	\item the $\ell^2$-Betti numbers of the pair $(N\setminus \nu(\Sigma),\Sigma_-)$ are zero, where $\nu(\Sigma)\cong\Sigma\times(-1,1)$ is the interior of a tubular neighborhood $\Sigma\times[-1,1]$ and $\Sigma_-$ is given by $\Sigma\times\left\{-1\right\}$.
\end{enumerate}	
\end{theorem}
\begin{rem}
	If one removes the assumption of $\Sigma$ being a decomposition surface, then $(1)$ still implies $(2)$.
\end{rem}

As an application of Theorem~\ref{thm:main2} we have the following theorem first proven by Friedl and L\"uck with different methods \cite{FL18}.
\setcounter{mytheorem}{\value{lemma}}
\begin{theorem}\label{thm:haupt:luckfried}
	Let $N$ be a connected compact oriented irreducible 3-manifold with empty or toroidal boundary and let $\phi\in H^1(N;\Z)$ be a primitive cohomology class. We write $N_{\ker \phi}\to N$ for the cyclic covering corresponding to $\ker \phi$. We have
	\[ b^{(2)}_1 (N_{\ker\phi}) = x_N( \phi ). \]
\end{theorem}
Here $x_N$ denotes the Thurston norm on $H^1(N;\Z)$ (see Section~\ref{sec:Definitions} for a definition).

Another application of Theorem~\ref{thm:main} is that for a taut sutured manifold the $\ell^2$-torsion is well defined.
\begin{cor}\label{cor:l2torsion}
Let $(M,R_+,R_-,\gamma)$ be a taut sutured manifold, then the pair $(M,R_-)$ is $\ell^2$-det-acyclic and hence the $\ell^2$-torsion $\rho^{(2)}(M,R_-)$ is well defined.
\end{cor} 
We refer to ~\cite[Definition 3.91]{Lu02} for the definitions of $\ell^2$-det-acyclic and $\ell^2$-torsion.
\begin{proof}[Proof of Corollary~\ref{cor:l2torsion}]	This follows from Theorem~\ref{thm:main}, Lemma ~\ref{lem:classg}, and ~\cite[Theorem 1.11]{Sch01}.
\end{proof}
If $M$ is a complete hyperbolic  3-manifold of finite volume, then by a result of L\"uck and Schick \cite[Theorem 0.7]{LS99} the $\ell^2$-torsion of $M$ is defined and one has
\[
\rho^{(2)}(M)=\frac{-1}{6\pi}\cdot\operatorname{Vol}_\mathbb{H}(M).\]

This result together with our corollary raises the following question.
\begin{ques}
	What topological and geometrical information of a taut sutured manifold $(M,R_+,R_-,\gamma)$ are contained in $\rho^{(2)}(M,R_-)$?
\end{ques} 
The author will pursue this question in a future paper with Ben-Aribi and Friedl.
Theorem ~\ref{thm:main} can be seen as an $\ell^2$-analogue of the following theorem by Friedl and T. Kim.
\begin{theorem}\cite[Theorem 1.1]{FK13} 
	Let $(M,\gamma)$ be a connected irreducible balanced sutured manifold with $M\neq S^1\times D^2$ and $M \neq D^3$. Then $(M,\gamma)$ is taut if and only if $H_*^{\alpha} (M,R_-;\C^k) = 0$ for some unitary representation $\alpha\colon\pi_1(M)\to U(k)$.
\end{theorem}

\subsection*{Outline of the content}
The paper is organised as follows. In Section \ref{sec:Definitions} we review the basic definitions and introduce our notation. In Section \ref{sec:basics} we discuss some basic properties of $\ell^2$-Betti numbers and sutured manifolds before we prove the main result in Section \ref{sec:proof}. In Section~\ref{sec:app} we show how Theorem~\ref{thm:haupt:luckfried} follows from Theorem~\ref{thm:main2}.
\subsection*{Acknowledgements} The author gratefully acknowledges the support provided by the SFB 1085 ‘Higher Invariants’ at the University of Regensburg, funded by the DFG. Moreover, I would like to thank my advisor Stefan Friedl for his many suggestions and his helpful advice.
\section{Definitions}\label{sec:Definitions}
\subsection{Norm minimizing surfaces and quasi-fibers}
Let $M$ be an oriented irreducible $3$-manifold and $A\subset \partial M$ a subsurface. The \emph{Thurston norm} is defined by
\begin{align*}
x_M\colon H_2(M,A;\Z)&\longrightarrow \Z\\
\sigma&\longmapsto \min\left\{\chi_-(S)\ \left|\ \begin{array}{ll}
[S]=\sigma 
\text{ and}\text{ $S$ is properly}\\
\text{embedded i.e.\ $\partial S=S\cap A$}
\end{array}\right.\right\}. 
\end{align*}
This map extends to a semi norm on $H_2(M,A;\R)$ by a result of Thurston~\cite{Th86}.

Recall that an embedded surface $S\hookrightarrow M$ is called \emph{incompressible} if for all choices of base points the inclusion map induces a monomorphism on the fundamental groups.
We call a properly embedded surface $S$ in $(M,A)$ \emph{Thurston norm minimizing} if $S$ is incompressible, $x_M([S,\partial S])=\chi_-(S)$ and $S$ has no disk or sphere components. The condition for a Thurston norm minimizing surface to be incompressible is not very restrictive. For example if the surface doesn't contain a torus or annulus component, then it is automatically satisfied (\cite[Chapter 3 C.22]{AFW15}).
\begin{examp}
	Thurston has shown that if $F\subset M$ with $\chi(F)\leq 0$ is a fiber of a fibration $M\to S^1$, then $F$ is Thurston norm minimizing in $(M,\partial M)$.
\end{examp}
We will also make use of the following theorem due to Gabai.
\begin{theorem}\cite[Corollary 6.13]{Ga83}\label{thm:thurstonnormmult}
	Let $\map{p}{N}{M}$ be a finite cover of a connected compact orientable $3$-manifold $M$ and $S$ a Thurston norm minimizing surface in $(M,\partial M)$. Then $p^{-1}(S)$ is a Thurston norm minimizing surface in $(N,\partial N)$. 
\end{theorem}

\begin{defn}
We call a properly embedded surface $S$ in $(M,\partial M)$ a \emph{quasi-fiber} if the following two conditions are satisfied.
\begin{enumerate}
	\item The surface $S$ is Thurston norm minimizing,
	\item and there exists a fibration over the circle with fiber $F$ such that \[x_M( [F])+x_M([S])=x_M([F]+[S]).\]
\end{enumerate}
\end{defn}
\begin{rem}
A surface is a quasi-fiber if and only if it is Thurston norm minimizing and the corresponding class lies in the boundary of a fiber cone in $H_2(M,\partial M;\R)$ ~\cite{Th86}.
\end{rem}

\subsection{Sutured manifolds} If $M$ is an oriented manifold then we endow $\partial M$ with the orientation coming from the outwards-pointing normal vectors.
 A sutured manifold $(M,R_+,R_-,\gamma)$ is a compact oriented 3-manifold with a decomposition of its boundary
\[\partial M = R_+\cup \gamma \cup -R_-, \]
into oriented submanifolds such that
\begin{enumerate}
	\item $\gamma$ is a collection of disjoint embedded annuli or tori,
	\item $R_+ \cap R_-=\emptyset$,
	\item if $A$ is an annulus component of $\gamma$, then $R_-\cap A$ is a boundary component of $A$ and of
	$R_-$, and similarly for $R_+\cap A$. Furthermore, $[R_+\cap A] = [R_-\cap A]\in H_1(A; \Z)$
	where we endow $R_\pm\cap A$ with the orientation coming from the boundary of the oriented manifold $R_\pm$.
\end{enumerate}
We call a sutured manifold $M$ \emph{taut}, if $M$ is irreducible and $R_+$ and $R_-$ are Thurston norm minimizing viewed as properly embedded surfaces in $(M,\gamma)$. We call a sutured manifold \emph{balanced} if $\chi_-(R_+)=\chi_-(R_-)$.
\begin{rem}
In our definition of a taut sutured manifold we demand $R_\pm$ to be incompressible, which differs from the convention most other authors choose. Our convention just rules out notorious counterexamples in the case that $M=S^1\times D^2$.
\end{rem}

An example of a balanced taut sutured manifold is given by the {\em product sutured manifold}
\[ \left(R\times [-1,1], R\times 1, R\times -1, \partial R\times[-1,1]\right),\]
where $R$ is a surface with $R\not\cong S^2$. Another example is given by an irreducible 3-manifold $N$ with empty or toroidal boundary where the sutured manifold structure is given by $\gamma=\partial N$.

\subsection{Sutured manifold decomposition} In \cite{Ga83} Gabai introduced the notation of a sutured manifold decomposition which we now recall. Let $(S,\partial S)$ be a properly embedded oriented surface in a sutured manifold $(M,R_+,R_-,\gamma)$. We call $S$ a \emph{decomposition surface} if $S$ is transversal to $R_\pm$ and for every connected component $c\in S\cap\gamma$ one of the following holds.

\begin{enumerate}
	\item $c$ is a properly embedded non-separating arc,
	\item $c$ is a simple closed curve in an annulus component of $\gamma$ which is homologous to $[R_-\cap A]\in H_1(A;\Z)$,
	\item $c$ is a homotopically non-trivial curve in a torus component $T$ of $\gamma$ and if $c'$ is another curve in $S\cap T$, then $c$ and $c'$ are homologous in $T$.     
\end{enumerate}
\begin{figure}
	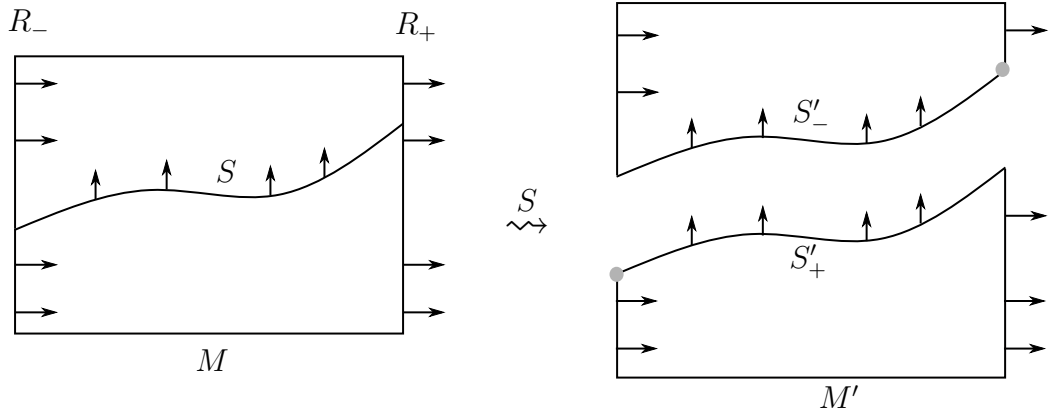
	\caption{This is a schematic picture of a sutured decomposition $M\sutdec{S} M'$. One cuts along a surface $S$ but has to keep track of the orientations. The grey dots denote the sutures.}\label{fig:sutdec}
\end{figure}
Given a decomposition surface $S$ we define the \emph{sutured decomposition} along $S$ by
\[
	(M,R_-,R_+, \gamma) \overset{S}{\rightsquigarrow} (M',R_-',R_+', \gamma ')
\]
where 
\begin{align*}
	M'&=M\setminus S\times (-1,1), \\
	\gamma' &= (\gamma \cap M')\cup \overline{\nu(S'_+\cap R_-)} \cup \overline{\nu(S'_-\cap R_+)},\\
	R'_+ &= ((R_+\cap M')\cup S'_+)\setminus \interior \gamma ',\\
	R'_- &= ((R_-\cap M')\cup S'_-)\setminus \interior \gamma '.
\end{align*}
Here $S'_+$ (resp. $S'_-$) is the outward-pointing (resp.\ inward-pointing) part of $S\times \{-1,1\}\cap M'$ (See Figure~\ref{fig:sutdec}). 
In the rest of the paper we suppress $R_\pm$ from the notation and abbreviate $(M,R_+,R_-,\gamma)$ to $(M,\gamma)$ assuming that $R_\pm$ is clear from the context. 
We make use of the following elementary lemma rather frequently.
\begin{lemma}\label{lem:closeddec}
	Let $(N,\emptyset,\emptyset,\partial N)$ be a sutured manifold (i.e.\ $N$ has empty or toroidal boundary) and $S$ a Thurston norm minimizing decomposition surface in $N$, then $N'$ defined by $N\sutdec{S} N'$ is a taut sutured manifold.
\end{lemma} 
We also need the following lemma from the theory of sutured manifold decomposition due to Gabai.
\begin{lemma}\label{lem:sutdec}
	Let $N$ be connected irreducible oriented closed 3-manifold with empty or torodial boundary. Let $S$ and $F$ be Thurston norm minimizing decomposition surfaces with $x_N([S])+x_M([F])=x_N([F]+[S])$. We assume that $S$ and $F$ are in general position such that the number of components of $S\cap F$ is minimal. Denote by $N'$ the sutured manifold obtained by $N\sutdec{S}N'$, then $F':=F\cap N'$ is a decomposition surface for $N'$. Moreover, $N'$ and $N''$ are taut sutured manifolds, where $N''$ is given by $N'\sutdec{F'} N''$. One also has a commutative diagram of taut sutured manifold decompositions
	\[
	\begin{tikzcd}
		&N\arrow[ld,squiggly,"F"]\arrow[d,squiggly,"F\oplus S"]\arrow[rd,squiggly,"S"] \\
	N\setminus\nu(F)\arrow[dr,squiggly,"S'"]& N\setminus\nu(F\oplus S)\arrow[d,squiggly,"C"] &N'\arrow[ld,squiggly,"F'"] \\
		& N'',
	\end{tikzcd}
	\] 
	 where $F\oplus S$ is the oriented cut and paste sum (see Figure~\ref{fig:orientedsum}), $S'= S\cap (N\setminus\nu(F))$, and $C=C_1,\ldots C_n$ is a disjoint union of annuli and disks. 
\end{lemma}
The main idea in the proof of that lemma is the following. The assumptions on $S$ and $F$ ensure that $F\oplus S$ is Thurston norm minimizing. Note that $M$ defined by $N\sutdec{F\oplus S} M$ is a taut sutured manifold, because $F\oplus S$ is Thurston norm minimizing. The same is true for $N'$. Now one can obtain $N''$ from $M$ by a decomposition surface only consisting of annuli or disks (see for example~\cite[Lemma 3.4]{FK14}). Then $N''$ is taut by \cite[Lemma 3.12]{Ga83}. 
We refer to \cite[Section 3]{Ga83} for more details.   
\begin{figure}
	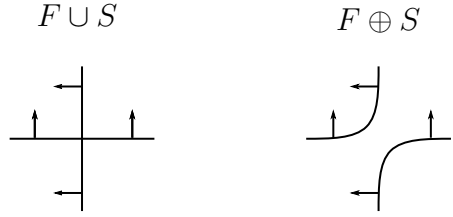
	\caption{Two oriented embedded surfaces $S$ and $F$ may intersect in circles or intervals. The above picture shows how to remove the intersection to get an oriented embedded surface which represents the homology class $[S]+[F]$.}\label{fig:orientedsum}
\end{figure}

Later we will need the double $DM(\gamma)$ of sutured manifold $(M,\gamma)$. It is defined by 
\[ DM(\gamma):= M \sqcup_{R_\pm} M. \]
This is a 3-manifold, which is closed or has toroidal boundary. It is also a sutured manifold with the sutured structure given by $(DM(\gamma),\emptyset,\emptyset,\partial DM(\gamma))$. Note that $R_+\cup R_-$ is a decomposition surface in $DM(\gamma)$ and if one decomposes $DM(\gamma)$ along $R_+\cup R_-$ one obtains two copies of $M$ as sutured manifolds. With this construction one can prove the following result which is analog to Theorem~\ref{thm:thurstonnormmult}.
\begin{prop}\label{prop: tautfinitecover}
	Let $(M,R_+,R_-,\gamma)$ be a sutured manifold. Let $\map{p}{N}{M}$ be a finite cover of $M$, then the preimage of $R_\pm$ and $\gamma$ under $p$ induces a sutured structure on $N$. If in addition $M$ is taut and $\gamma$ is incompressible, then the induced structure on $N$ is taut as well.
\end{prop}
\begin{proof}
The first assertion of the proposition is clear. Hence we only show that if $M$ is taut and $\gamma$ is incompressible, then $N$ is taut.  Under these two assumptions one has by a result of Gabai (\cite[Lemma 3.7]{Ga83}) that $R_-\cup R_+$ is Thurston norm minimizing in $H_2 ( DM(\gamma),\partial DM(\gamma);\Z)$. The proposition follows from Theorem~\ref{thm:thurstonnormmult}, Lemma~\ref{lem:closeddec}, and the fact that a finite cover of an irreducible 3-manifold is again irreducible \cite[Theorem 3]{MSY82}.   
\end{proof}

\subsection{Incompressible decomposition surfaces in a product sutured manifold}
We quickly recall a result of Waldhausen \cite{Wa68}, which will be essential in the proof of Lemma~\ref{lem:quasifiber}. Let $F$ be a connected surface possibly with boundary and $F\not\cong S^2$, then we endow $F\times[-1,1]$ with the product sutured manifold structure $(F, F_+, F_-, \partial F\times I)$, where $F_+= F\times\left\{1\right\}$ and $F_-= F\times\left\{-1\right\}$. We denote by $p\colon F\times[-1,1]\to F$ the canonical projection. A properly embedded surface $S$ in $F\times [-1,1]$ is called $\emph{horizontal}$ if $p|_{S}$ is a homeomorphism onto its image.
\begin{rem} \label{rem:isotopic}
	A horizontal surface $S$ is by definition homoemorphic to a subsurface of $F$, so that we can view it as an embedding $S\to F\times[-1,1], x\mapsto (x, h(x))$. Therefore $S$ is isotopic to a subsurface of $F_-$ (resp. $F_+$) by pushing (resp. lifting) the interval factor.
\end{rem}
Given a connected incompressible decomposition surface $S$ in $F\times[-1,1]$ there are evidently two possibilities:
\begin{enumerate}
 \item $S$ intersects $F_+$ and $F_-$,
 \item $S\cap F_+=\emptyset$ or $S\cap F_-=\emptyset$.
\end{enumerate}
Waldhausen~\cite[Proposition 3.1]{Wa68} showed that in the second case if $S\cap F_+=\emptyset$ (resp. $S\cap F_-=\emptyset)$, then $S$ is ambient isotopic to a horizontal surface via an ambient isotopy fixing $F_\pm$.
\begin{rem}\label{rem:parallel}
	The second case also includes the possibility that $S$ intersects neither $F_+$ nor $F_-$. In this situation $S$ is ambient isotopic to $F\times\left\{t\right\}$ for $t\in(-1,1)$.
\end{rem}
Later we need the following lemma, which easily follows from Waldhausen's result.

\begin{lemma}\label{lem:waldhausen}
Let $N\not\cong S^1\times S^2$ be a connected 3-manifold which fibers over $S^1$ with fiber $F$. We fix an embedding of $F$ and an identification $N\setminus\nu(F)\cong F\times[-1,1]$. Let $\Sigma$ be an incompressible not necessarily connected surface in $N$ such that $F\cap \Sigma$ has minimal number of connected components compared to all other embeddings in the isotopy class of $\Sigma$. Then one of the following holds:
\begin{enumerate}
	\item $\Sigma$ consists of parallel copies of surfaces, all isotopic to $F$,
	\item every component of $\Sigma':=\Sigma\cap N\setminus\nu(F)$ intersects $F_+$ and $F_-$. 
	\end{enumerate} 
\end{lemma}
\begin{proof}	
A standard argument using the irreducibility of $N$ and our hypothesis on $F\cap\Sigma$ shows, that $F\cap \Sigma$ is incompressible and hence $\Sigma'$ is incompressible in $F\times[-1,1]$. If $\Sigma'$ does not intersect $F_+\cup F_-$, then by Remark ~\ref{rem:parallel} every component of $\Sigma'$ is ambient isotopic to $F\times\left\{t\right\}$ for some $t\in(-1,1)$. Therefore $\Sigma$ consist of parallel copies of $F$. Now let $C$ be connected component of $\Sigma'$, which intersects $F_-$ at least once but does not intersect $F_+$. Then there is an ambient isotopy fixing $F_\pm$, which makes $C$ into a horizontal surface. Since this isotopy fixes $F_\pm$, this isotopy extends to an isotopy of $\Sigma$ in $N$. If we further assume that $C$ is an innermost among such a connected component, then one can use the isotopy from Remark~\ref{rem:isotopic} to remove the intersection component which corresponds to $C\cap F_-$. But this contradicts our assumptions on $F\cap\Sigma$ and hence $C$ has to intersect $F_+$. The same argument with the roles of $F_+$ and $F_-$ interchanged proves the lemma.
\end{proof}

\subsection{$\ell^2$-Betti numbers}
Here we introduce $\ell^2$-Betti numbers and discuss some of their basic properties. For more details on $\ell^2$-Betti numbers and proofs of the facts stated here we refer to \cite[Chapter 1]{Lu02}.

Let $G$ be a group. We endow $\C[G]$ with the pre-Hilbert space structure for which $G$ is a orthonormal basis and denote by $\ell^2(G)$ the closure of $\C[G]$ with respect to the norm induced from the scalar product. Multiplication with elements in $G$ induces an isometric left $G$-action on $\ell^2(G)$. We define the group von Neumann algebra $\mathcal{N}(G)$ to be the set of all bounded linear operators from $\ell^2(G)$ to itself which commute with this left action. 

\begin{defn}
	Let $\widehat{X}$ be a $CW$-complex on which a group $G$ acts co-compactly, freely and cellularly. Then $C^{CW}_\ast (\widehat X)$ is a finite free $\Z[G]$-module and we can consider the chain complex
	\[\CC_\ast(\widehat{X};\mathcal{N}(G)):=\ell^2(G)\otimes_{\Z[G]} C^{CW}_\ast (\widehat X) \] with the $\ell^2$-homology defined by
	\[\HH_i(\widehat{X};\mathcal{N}(G)):= \ker (\id\otimes \partial_{i}) / \overline{\image (\id\otimes \partial_{i+1})},\]
	where $\overline{\image (\id\otimes \partial_{i-1})}$ denotes the closure of $\image (\id\otimes \partial_{i-1})$ with respect to the Hilbert space structure on $\ell^2(G)$.
	The $\ell^2$-Betti numbers are then given by 
	\[ b^{(2)}_i(\widehat{X};\mathcal{N} (G))=\dimNG{G}\HH_i(\widehat{X};\mathcal{N}(G)), \]
	where $\dimNG{G}$ denotes the von Neumann dimension \cite[Definition 1.10]{Lu02}.
	Moreover, we define the $\ell^2$-Euler characteristic by 
	\[\chi^{(2)}(\widehat{X};\mathcal{N}(G)):=\sum_{i\in\N_0}(-1)^i\cdot \BB_i(\widehat{X};\mathcal{N}(G)). \]
\end{defn}
\begin{rem}
These numbers only depend on the homotopy type of $\widehat{X}/G$ and are especially independent of the choice of CW-structure. 	
\end{rem}

The von Neumann dimension behaves very similarly to the ordinary dimension. For example it is additive under weakly short exact sequences. With this property one easily shows that $\ell^2$-Euler characteristic and Euler characteristic are related by
\begin{align}\label{eq:eulerchar}
	\chi^{(2)}(\widehat{X};\mathcal{N}(G))=\chi(\widehat{X}/G).
\end{align}
 One of the main features of $\ell^2$-Betti numbers is given by the next proposition.
\begin{prop}\label{prop:multiplicativ}
	Let $\widehat{X}$ be a $CW$-complex on which a group $G$ acts co-compactly, freely and cellularly. Let $H\triangleleft G$ be a finite index subgroup. Then 
	\[\BB_i(\widehat{X};\mathcal{N}(H))=[G:H] \cdot \BB_i(\widehat{X};\mathcal{N}(G)).\]
\end{prop}
In this paper we deal with two special cases of groups acting on CW-complexes which therefore obtain a shortened notation. If $X$ is a finite connected $CW$-complex with fundamental group $\pi$, we denote by the corresponding universal cover $p\colon\widetilde{X}\to X$ and then we define the $\ell^2$-homology of $X$ by:
\[ \HH_i (X):=\HH_i(\widetilde{X};\mathcal{N}(\pi)).\]
We write $\BB_i(X)=\dimNG{\pi} H_i^{(2)}(X)$ and in the case that $X=X_1\sqcup\ldots\sqcup X_n$ has several connected components we set $b_i^{(2)}(X):=\sum_{k=1}^{n}b_i^{(2)}(X_k)$.

Moreover, if $Y\subset X$ is a subcomplex, then we set $\widetilde{Y}:=p^{-1}(Y)$ and define
\begin{align*}
\HH_i (Y \subset X)&:=\HH_i(\widetilde{Y};\mathcal{N}(\pi)),\quad b_i^{(2)}(Y\subset X):=\dimNG{\pi} \HH_i(\widetilde{Y};\mathcal{N}(\pi)), \\
\HH_i (X,Y)&:=\HH_i(\widetilde{X}, \widetilde{Y};\mathcal{N}(\pi)),\quad \BB_i(X,Y):= \dimNG{\pi} \HH_i(\widetilde{X}, \widetilde{Y};\mathcal{N}(\pi)).
\end{align*} 
With this pullback of coefficients one has Mayer-Vietoris sequences and the long exact sequence associated to a pair. Proposition~\ref{prop:multiplicativ} and Equation~(\ref{eq:eulerchar}) holds equally in the relative case. Moreover, if the inclusion $Y\to X$ induces a monomorphism on the fundamental group for any choice of base-point, then one has by the induction principle \cite[Theorem 1.35(10)]{Lu02}:
\begin{equation}\label{eq:induction}
b_i^{(2)}(Y\subset X)=b_i^{(2)} (Y).
\end{equation}
If $\phi\colon\pi\rightarrow \Z$ is an epimorphism, we denote by $\widehat{X}$ the covering corresponding to $\ker\phi$ and introduce the notation
\[H_\ast^{\phi,(2)}(X):=\HH_\ast(\widehat{X};\,\mathcal{N}(\Z) ). \]
Let $\langle t\rangle\cong \mathbb{Z}\cong\pi/\ker \phi$ denote a generator, then by \cite[Lemma 1.34]{Lu02} one has
\begin{align}\label{eq:dimfielddimneu}
\dimNG{\Z}H_i^{\phi,(2)}(X)=\dim_{\C(t)} H_i(X;\C(t)^{\phi} ),
\end{align}
where $H_i(X;\C(t)^{\phi})$ denotes the homology of the chain complex $\C (t)\otimes_{\Z[t^{\pm}]} C^{CW}_\ast (\widehat{X})$.

\subsection{Approximation of $\ell^2$-Betti numbers}
In this paragraph we recall the L\"uck-Schick approximation result of $\ell^2$-Betti numbers. In order to state the theorem and that it applies in our situation we need some preliminaries. 
\begin{defn}
	Let $\mathcal{G}$ be the smallest class of groups which contains the
	trivial group and is closed under the following processes:
\begin{enumerate}
	\item\label{en:amenable} If $H<\pi$ is a normal subgroup such that $H\in\mathcal{G}$ and $\pi/H$ is amenable then $\pi \in\mathcal{G}$.
	\item If $\pi$ is the direct limit of a directed system of groups $\pi_i\in\mathcal{G}$, then $\pi \in \mathcal G$.
	\item\label{en:invlim} If $\pi$ is the inverse limit of a directed system of groups $\pi_i\in\mathcal{G}$, then $\pi \in \mathcal G$.
	\item\label{en:subgroup} The class $G$ is closed under taking subgroups. 
\end{enumerate}
\end{defn}

The precise definition of an amenable group doesn't play a role for this article. We only need that finite groups are amenable. This is sufficient to prove the following lemma.
\begin{lemma}\label{lem:classg}
	Every fundamental group of a compact 3-manifold lies in $\mathcal{G}$.
\end{lemma}
\begin{proof}
	By fact (\ref{en:amenable}) every finite group lies in $\mathcal{G}$. Then by fact (\ref{en:invlim}) the profinite completion of a group lies in $\mathcal{G}$. Since residually finite groups are subgroups of their profinite completation we have by fact (\ref{en:subgroup}) that all residually finite groups are in $\mathcal{G}$. So the lemma follows from the fact that all 3-manifold groups are residually finite \cite[Chapter 3 C.29]{AFW15}.
\end{proof}

We are now able to state the approximation result of Schick which extended earlier results by L\"uck~\cite{Lu94}.
\begin{theorem}\cite[Theorem 1.14]{Sch01}\label{thm:appschick}
	Let $X$ be a CW-complex with a free cellular and co-compact action of a group $G$ in the class $\mathcal{G}$.  Let $G= G_1\supset G_2\supset \ldots$ be a nested sequence of normal subgroups such that $\cap_{i\in \N} G_i=\left\{e\right\}$. Denote by $X_i=X/G_i$ and by $\Gamma_i=G/G_i$ the corresponding quotients. Further assume that $\Gamma_i\in\mathcal{G}$ for all $i\in\N$, then one has for all $p\in\Z$
	\[	\lim_{i\to\infty} \BB_p (X_i;\mathcal{N}(\Gamma_i)) = \BB_p(X;\mathcal{N}(G)). \]
\end{theorem}
Note that the action of each $\Gamma_i$ on $X_i$ is co-compact, free and cellular.

\section{Basics of $\ell^2$-betti numbers of sutured manifolds}\label{sec:basics}
In this section we prove the equivalence of statements (2) and (3) in Theorem~\ref{thm:main}. Moreover, we prove a vanishing criteria for the $\ell^2$-homology of a cyclic covering of a sutured manifold.

As mentioned in the beginning we state every result only for the pair $(M,R_-)$, but by \Poincare Lefschetz duality (see theorem below) all results hold equally for the pair $(M,R_+)$.
\begin{theorem}\cite[Theorem 1.35(3)]{Lu02}\label{thm:poindual}
	Let X be a compact and connected $n$-manifold together with a decomposition $\partial X=Y_1\cup Y_2$, where $Y_1$ and $Y_2$ are compact $(n-1)$-dimensional submanifolds of $\partial X$ with $\partial Y_1 = \partial Y_2$. Then 
	\[ b_{n-i}^{(2)}(X,Y_1) = b_{i}^{(2)}(X,Y_2). \]
\end{theorem}
\begin{rem}
L\"uck only discusses the case $Y_1=\partial X$ but the same proof he gives works also in the above setting (compare \cite[Chapter 5.9 Exercise 3]{Br93}).
\end{rem}
We also use this duality to obtain a general result about $\ell^2$-Betti numbers of balanced sutured manifolds.
\begin{prop}\label{prop:balanced}
	Let $M$ be a balanced sutured manifold with infinite fundamental group, then $b_1^{(2)}(M,R_{-})=b_2^{(2)}(M,R_{-})$, and $b_1^{(2)}(M,R_{-})=0$ implies for all $i\in \N$ $b_i^{(2)}(M,R_{-})=0$.
\end{prop}
\begin{proof}
 One has $b_0^{(2)}(M,R_\pm) = 0$, because $\pi_1(M)$ is infinite~\cite[Theorem 1.35(8)]{Lu02}  and by \Poincare Lefschetz duality (Theorem \ref{thm:poindual}) $b_3^{(2)}(M,R_\mp) = 0$, too. Since $M$ is balanced, we have
 \[\chi(R_-)=\frac{\chi(R_+\cup\gamma\cup R_-)}{2}=\frac{\chi(\partial M)}{2}=\chi(M) \]
 and hence $\chi(M,R_-)=0$. By the relation between $\ell^2$-Euler characteristic and Euler characteristic (see Equation (\ref{eq:eulerchar})) we obtain \[\chi^{(2)} (\widetilde{M},\widetilde{R}_-;\mathcal{N}(\pi_1(M))=\chi(M,R_-)= 0\] and thus $b_1^{(2)}(M,R_-) =b_2^{(2)}(M,R_-)$.
\end{proof}
This gives us already the equivalence of (2) and (3) in the main result (Theorem \ref{thm:main}):
\begin{cor}\label{cor:equiv2n3}
Let $(M,\gamma)$ be an irreducible balanced sutured manifold. Assume that $\gamma$ is incompressible and $\pi_1(M)$ is infinite. Then $\BB_*(M,R_-)=0$ if and only if $\HH_1(R_-\subset M)\to\HH_1(M)$ is a monomorphism.
\end{cor}
\begin{proof}
	We look at the long exact sequence in $\ell^2$-homology of the pair $(M,R_-)$. Note that $\HH_2(M)=0$, because $M$ is irreducible and $\pi_1(M)$ is infinite~\cite[Theorem 4.1]{Lu02}. Therefore the sequence becomes
	\[\begin{tikzcd}[column sep=small]
		0\arrow[r]&\HH_1(M,R_-)\arrow[r]& \HH_1(R_-\subset M)\arrow[r]& \HH_1(M)\arrow[r]&\ldots \ . 
	\end{tikzcd}\]
Now the corollary follows from Proposition~\ref{prop:balanced} and the fact that the von Neumann dimension is zero if and only if the module is zero~\cite[Theorem 1.12(1)]{Lu02}.
\end{proof}
We end this section with two technical lemmas, which are the cornerstone for the proof of Theorem~\ref{thm:lastthm}. The first of these two lemmas is a vanishing result of certain $\ell^2$-Betti numbers and the second lemma gives a sufficient criteria two apply the first one.
\begin{lemma}\label{lem:vanishing}
	Let $(M,\gamma)$ be a connected sutured manifold and let $\phi\in H^1(M;\Z)$ be non-trivial. If there is a decomposition surface $S$ such that the class $[S]\in H_2(M,\partial M;\Z)$ is \Poincare dual to $\phi$ and $M \sutdec{S} M'$ results in a product sutured manifold $M'$, then
\[ H^{\phi,(2)}_*(M,R_-) = 0. \]
\end{lemma}
Before giving the proof it is worth discussing a simple case which motivates the proof. Namely, if $N$ is sutured manifold with empty or toroidal boundary where the sutured manifold structure is given by $\gamma=\partial N$. In this case one has $R_-=\emptyset$. Moreover, if $S$ is  a decomposition surface $S$ such that the decomposition $N\sutdec{S}N'$ results in a product sutured manifold, then $S$ is a fiber of a fibration of $N$ over $S^1$. The cover corresponding to $\ker (\phi)$ is homeomorphic to $S\times\R$ and hence $H_\ast(N;\C(t)^{\phi})=H_\ast(S\times\R;\Z)\otimes_{\Z[t^\pm]} \C(t)=0$.
\begin{proof}[Proof of Lemma~\ref{lem:vanishing}]
%
We have two canonical embeddings of $S$ into the boundary of $M'$ which we denote by $\map{i_\pm}{S}{M'}$. Moreover, we denote by $S_\pm$ the images $i_\pm(S)$.

Since $\phi$ is the \Poincare dual of $S$, the cyclic cover $p\colon\widehat{M}\to M$ corresponding to $\ker\phi$ can be described by
	\[ \widehat{M}= (M'\times \Z)/\sim, \]
	where $(S_-,i)$ is glued to $(S_+,i+1)$ in the obvious way (see Figure~\ref{fig:cycliccover}).
The deck transformation group acts on the $\Z$-factor. We denote by $t$ the generator of this action i.e.\ $t\cdot(x,i)=(x,i+1)$ for all $x\in M'$. 
\begin{figure}
	\scalebox{.75}{\includegraphics{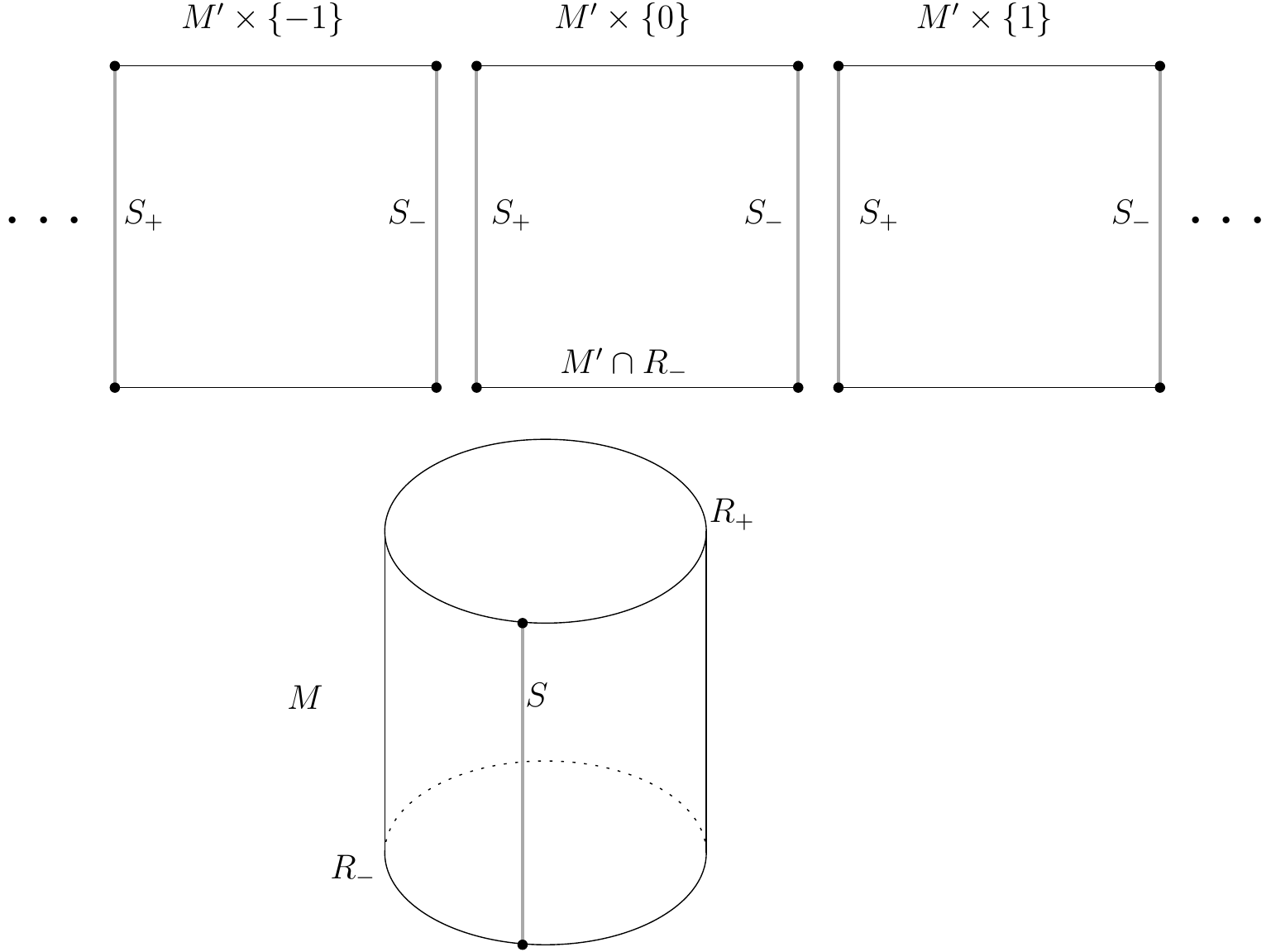}}
	\caption{This is a schematic picture one dimension reduced. It illustrates the cyclic cover $\widehat{M}$ of $M$ corresponding to the kernel of $\phi$, where $\phi$ is \Poincare dual to an embedded surface $S$.}\label{fig:cycliccover}
\end{figure}

We decompose $\widehat{M}$ into the subsets $\left\{M'\times\{i\}\right\}_{i\in\Z}$ and we set $\partial_- S=S\cap R_-$. Given a CW-complex $X$ we abbreviate $C^{CW}_*(X;\Z)$ and $H_*(X;\Z)$ to $C_*(X)$ and $H_*(X)$ so that we can express the Mayer-Vietoris sequence in cellular homology with respect to the above decomposition in the following way: 
\[\begin{tikzcd}[column sep=tiny]
0\arrow[r] & C_*(S,\partial_- S)\otimes_{\Z} \Z[t^\pm] \arrow[r,"i_--t\cdot i_+"] &[2.1em] C_*(M',M'\cap R_-)\otimes_{\Z} \Z[t^\pm] \arrow[r]& C_*(\widehat{M},\widehat{R}_-) \arrow[r] & 0.
\end{tikzcd}
\]
This yields a long exact sequence in homology
\[\begin{tikzcd}[column sep=tiny]
\ldots\hspace*{-0.3em}\arrow[r] &[0.01em]\hspace*{-0.3em} H_i(S,\partial_- S)\otimes_{\Z} \Z[t^\pm] \arrow[r,"i_--t\cdot i_+"] &[2.1em]H_i(M',M'\cap R_-)\otimes_{\Z} \Z[t^\pm]\hspace*{-0.2em} \arrow[r]&[0.01em] H_i(\widehat{M},\widehat{R}_-)\hspace*{-0.3em}\arrow[r]&[0.01em]\hspace*{-0.5em} \ldots
\end{tikzcd}.	
\]
Next we show that the inclusion $i_-\colon(S,\partial_- S)\to(M',M'\cap R_-)$ induces an isomorphism on homology.

 By assumption $M'$ is a product sutured manifold i.e.\ $M'=R'_-\times I$ with $R'_-=(M'\cap R_-)\cup S_-$. Therefore, we get by homotopy invariance \[H_*(M',M'\cap R_-)=H_*(R'_-\times I,M'\cap R_-)\isofrom H_*(R'_-,M'\cap R_-) \] and by excision 
 \[H_*(R'_-,M'\cap R_-)=H_*((M'\cap R_-)\cup S_-, M'\cap R_-)\xleftarrow{i_-} H_*(S,\partial_- S). \]
We refer to Figure~\ref{fig:isomorphism} for an illustration of this argument.
\begin{figure}
	\scalebox{0.75}{\includegraphics{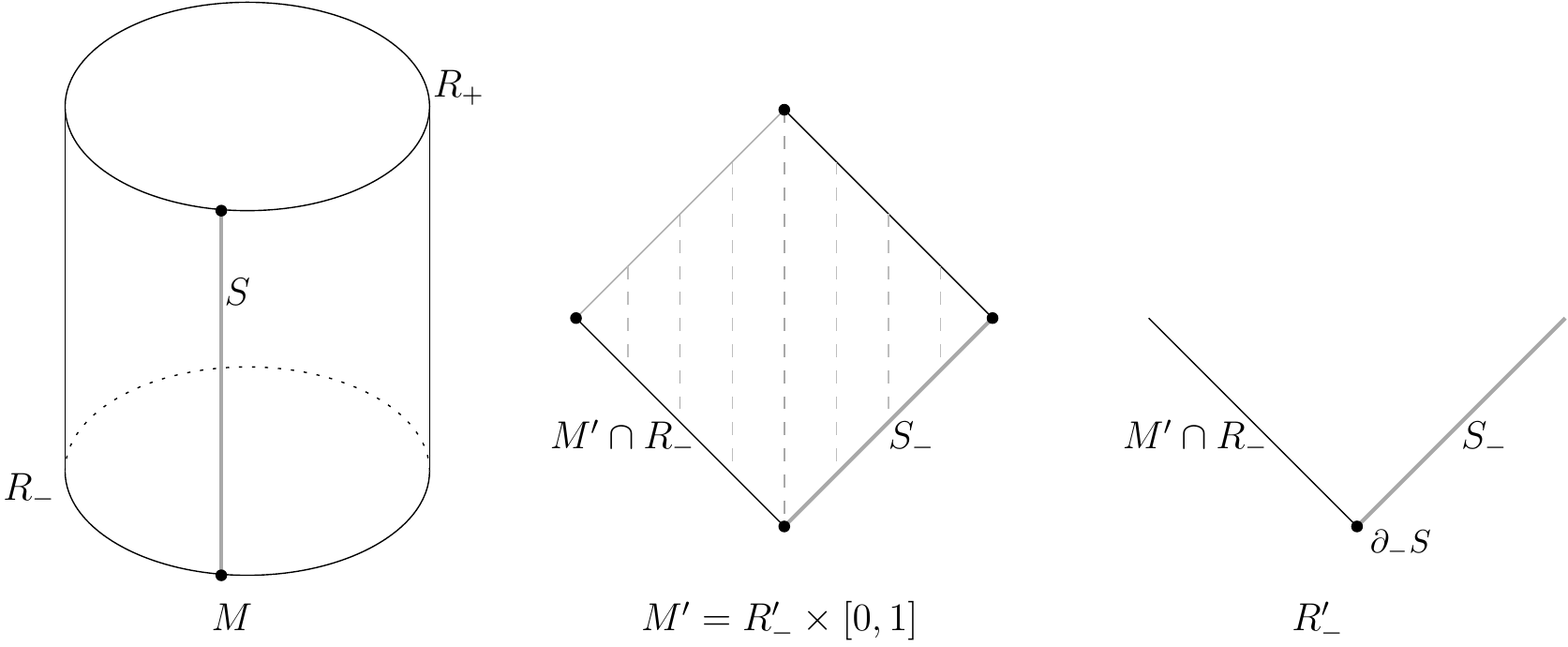}}
	\caption{A schematic picture one dimension reduced. The sutured decomposition $M\sutdec{S} M'$ results in a product sutured manifold. The dashed lines show the $[0,1]$-factor of the product. By homotopy invariance and excision one has an isomorphism $H_*(S,\partial_-S)\xrightarrow{i_-} H_*(M',M'\cap R_-)$. }\label{fig:isomorphism}
\end{figure}

We now continue with the rest of the proof. Because $H_*(M',M'\cap R_-)\cong H_*(S,\partial_- S)$ is free abelian it makes sense to talk about the determinant. And since $i_-$ induces an isomorphism on homology we have $\det_\Z (i_- )\neq 0$. This of course implies that $\det_{\Z[t^\pm]} (i_--t\cdot i_+)\neq 0$. Therefore the map $i_--t\cdot i_+$ is invertible over $\C(t)$. Also note that we have $H_*(\widehat{M},\widehat{R}_-;\Z)=H_*(M,R_-;\Z[t^\pm]^{\phi})$.

Now we use the above sequence, the fact that $\C (t)$ is flat over $\Z[t^\pm]$ and that $i_--t\cdot i_+$ is invertible over $\C(t)$ to obtain 
$H_*(M,R_-;\C(t)^{\phi})=0$. Then Equation~(\ref{eq:dimfielddimneu}) yields $H^{\phi,(2)}_*(M,R_-) = 0.$ 
%
%
%
\end{proof}

\begin{lemma}\label{lem:quasifiber}
Let $(N,\emptyset,\emptyset,\emptyset)$ be a taut sutured manifold and let $\Sigma$ be a quasi-fiber in $N$ which is not the fiber of a fibration. Then the sutured manifold $M$ obtained by $N\sutdec{\Sigma} M$ contains a decomposition surface $S$ such that $M\sutdec{S}M'$ results in a product sutured manifold $M'$ and the class $[S]\in H_2(M,\partial M;\Z)$ is non-trivial.
\end{lemma}

\begin{proof}
Since $\Sigma$ is a quasi-fiber there is a fibration of $N$ over $S^1$ with fiber $F$ and $\chi_N([F])+\chi_N([\Sigma])=\chi_N([F]+[\Sigma])$. We can assume that $F$ and $\Sigma$ are in general position such that the number of components of $\Sigma\cap F$ is minimal compared to all surfaces isotopic to $\Sigma$ and $F$.

We define $F':=F\cap M=F\setminus\nu(F\cap \Sigma)$ and make the following claim.
\begin{claim}
The sutured decomposition $(M,\Sigma_+,\Sigma_-,\gamma) \sutdec{F'} (M',R'_+,R'_-,\gamma')$ results in a product sutured manifold.
\end{claim}
We set $\Sigma'=\Sigma\cap (N\setminus\nu(F))$. By Lemma~\ref{lem:sutdec} we have that $M'$ is a taut sutured manifold and by the commutativity of the diagram in Lemma~\ref{lem:sutdec} we have that $N\sutdec{F}N\setminus\nu(F)\sutdec{\Sigma'} M'$. Moreover, $N\setminus\nu(F)\cong F\times[-1,1]$ is a product, since $F$ is a fiber of a fibration. The taut sutured manifold decomposition of a product sutured manifold is again a product sutured manifold (\cite[Remark 4.9(4)]{Ga83}) and hence $M'$ is a product.

It remains to show, that $[F']\in H_2(M,\partial M)$ is non-trivial. This follows directly from the next claim.
\begin{figure}
	\includegraphics{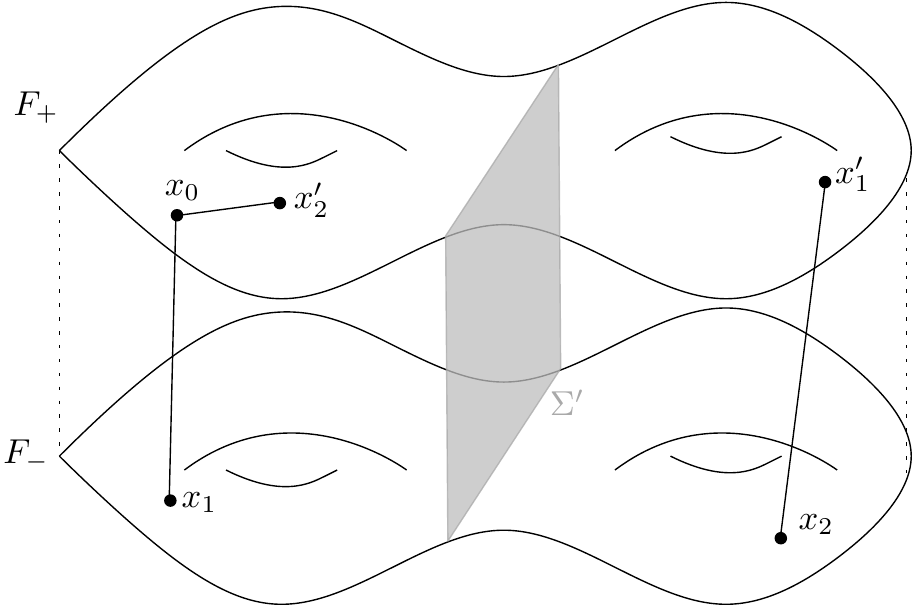}
	\caption{A schematic picture of how the loop $c$ in the proof of Proposition~\ref{prop:virtualvanishing} is constructed. We see the product surface $F\times[-1,1]$. The grey rectangle represents $\Sigma\cap F\times [-1,1]$. The points $x_i$ and $x'_i$ are identified by the monodromy in $N$. The concatenation of the paths gives a loop in $N$.    }\label{fig:closedloop}
\end{figure}

\begin{claim}
	 There is a closed curve $c$ in $N$, which doesn't intersect $\Sigma$ but has a positive intersection number with $F$.
\end{claim}
The curve $c$ is constructed as follows (see Figure~\ref{fig:closedloop}). We choose a point $x_0\in F_+\setminus\Sigma'$ and a path $p_0$ not intersecting $\Sigma'$ to $F_-$. Such a path always exists by Lemma~\ref{lem:waldhausen} and the assumption that $\Sigma$ is not a fiber of a fibration. The monodromy sends the endpoint of this path to a new point $x'_1$ on $F_+$ maybe in a different connected component of $M'$. We repeat this process to obtain another path $p_1$ connecting $x'_1$ with another point $x_2\in F_-$, which is sent to $x'_2\in F_+$ via the monodromy. Since there are only finitely many connected components of $M'$ we can after several iterations of this process join $x_n$ with $x_0$ in $F_+$ by a path $p$ not intersecting $\Sigma'$. All these paths patched together give a closed loop in $N$. This loop does not intersect $\Sigma$ but gives a positive intersection number with $F$.

Since $c$ does not intersect $\Sigma$ it is a loop in $M$ and since it has positive intersection number with $F$, the class $[F']\in H_2(M,\partial M)$ is non-trivial.
\end{proof}

\section{Proof of the Main Theorem}\label{sec:proof}
\subsection{$\ell^2$-acyclic implies taut}
The basic idea is that $\BB_2(M,R_-)$ is an upper bound for how far a sutured manifold $(M,\gamma)$ is away from being taut.
\begin{lemma}[Half lives, half dies]\label{lem:hlhd}
Let $W$ be a compact connected $(2k+1)$-dimensional manifold, then
\[\dimNG{\pi_1(W)} \ker \left(i_\ast\colon\HH_k(\partial W\subset W)\to\HH_k (W) \right)= \frac{1}{2}\cdot b_k^{(2)}(\partial W\subset W). \]
\end{lemma}
\begin{proof}
We write $G=\pi_1(W)$ and $\map{i_*}{\HH_k(\partial W)}{\HH_k(W)}$. Because of the additivity of the von Neumann dimension we get
	\[
	b_k^{(2)}(\partial W\subset W)= \dimNG{G} \ker (i_*) + \dimNG{G} \overline{\image (i_*)}.
	\]
	Therefore it is sufficient to prove that $\dimNG{G} \ker (i_*) = \dimNG{G} \overline{\image (i_*)}$. To show this one considers the long exact sequence in homology associated to the pair $(\partial W, W)$. This sequence can be decomposed into long exact sequences
	\[
	\begin{tikzcd}[column sep=small, row sep=small]
		\ldots\arrow[r]&\HH_{k+1}(W,\partial W)\arrow[r]\arrow[d]&\ker (i_*) \arrow[r]\arrow[d]& 0\arrow[d]\arrow[r]&\ldots \\
	\ldots\arrow[r]&\HH_{k+1}(W,\partial W)\arrow[r]\arrow[d]&\HH_k(\partial W\subset W)\arrow[r]\arrow[d]&\HH_k(W)\arrow[d]\arrow[r]&\ldots \\
	\ldots\arrow[r]&0\arrow[r]&\overline{\image (i_*)} \arrow[r]& \HH_k(W)\arrow[r]&\ldots\ .
	\end{tikzcd}
	\]
One has $b_{k+1+i}^{(2)}(W,\partial W)=b_{k-i}^{(2)}(W)$ and $b_{k+i}^{(2)} (\partial W\subset W) = b_{k-i}^{(2)} (\partial W\subset W)$ by \Poincare duality. So we see that the Euler characteristic of the upper and the lower exact sequence contains the same summands except for $\dimNG{G} \ker (i_*)$ and $\dimNG{G}\overline{\image (i_*)}$. But both sequences have zero Euler characteristic, because they are exact. Hence $\dimNG{G} \ker (i_*) = \dimNG{G} \overline{\image (i_*)}$.
\end{proof}

\begin{lemma}\label{lem:sepsurf}
	Suppose $(M,\gamma)$ is a connected irreducible sutured manifold with infinite fundamental group and $\gamma$ is incompressible. Then there exists a Thurston norm minimizing representative $S\subset M\setminus (R_+\cup R_-)$ of $[R_-]\in H_2(M,\gamma;\Z)$ such that $M$ cut along $S$ is the union of two disjoint (not necessarily connected) compact manifolds $M_\pm$ with $R_\pm\subset \partial M_\pm$.
\end{lemma}
\begin{proof}
	Let $T$ be a properly embedded surface homologous to $[R_-]$, with $\chi_-(T)=x_M([R_-])$ and the intersection of $T$ with $R_\pm$ is empty. We will show that a subsurface of $T$ has the desired properties.
	
	First we find a subsurface $S$ of $T$ which is Thurston norm minimizing. By the definition of Thurston norm minimizing we have to show that there is a subsurface $S\subset T$ which is homologous to $T$, which has the same complexity, does not have sphere or disk components, and is incompressible.
	
	Since $M$ is irreducible every sphere component of $T$ is null homologous and we can omit it. Now let $D\subset T$ be a disk component. From the definition one has $\partial D\subset \gamma$. Therefore we have to consider two cases. The first case that $\partial D$ is homotopically trivial in $\gamma$. In this case we can close the circle with another disk in $\gamma$ and obtain a sphere. By the assumption $M$ is irreducible and therefore this sphere bounds a $3$-ball. Hence the disk $D$ is homologically trivial in $H_2(M,\gamma;\Z)$. Therefore $[T]=[T\setminus D]$ and we define $S:=T\setminus D$. Note that $S$ has the same complexity as $T$. 
	The other case that $\partial D$ is non trivial in $\gamma$ can not occur because $\gamma$ is incompressible.
	That $S$ is indeed incompressible is a consequence of the loop theorem and the fact that it has minimal complexity. See also ~\cite[Chapter 3 C.22]{AFW15} or~\cite[Lemma 5.7]{Ca07}.
	
	It remains to show that $S$ separates the manifold $M$ into at least two disjoint parts.
	We have an intersection form:
	\[H_2(M,\gamma;\Z) \times H_1(M,R_+\cup R_-;\Z)\to \Z\]
	Let $p$ be a path from $R_+$ to $R_-$, then the intersection number of $[p]$ and $[R_-]$ is equal to 1. So every surface homologous to $[R_-]$ has to intersect $p$ at least once and therefore separates $R_-$ and $R_+$.
\end{proof}

\begin{lemma}\label{lem:inequaility}
	With the notation of the previous lemma and the additional assumption that $R_-$ is incompressible one has
	\[\frac{1}{2} \big(\chi(S)-\chi(R_+)\big)\leq b_2^{(2)}(M,R_-).  \]
\end{lemma}
\begin{proof}
	Applying Mayer-Vietoris on $U:=R_-\cup S$ and $V=\gamma '$ for the boundary $\partial M_- = R_+\cup \gamma ' \cup S$ we get an weak isomorphism \[H^{(2)}_1(\partial M_-\subset M)\cong H^{(2)}_1(R_-\subset M)\oplus H^{(2)}_1(S\subset M).\] Here we used that $\gamma'\subset \gamma$ is incompressible in $M$. We set $G=\pi_1(M)$ and will consider all $\ell^2$-homology with the coefficient system coming from $M$. Therefore, we drop ``$\subset M$'' from the notation. From Lemma~\ref{lem:hlhd} applied to the boundary of $M_-$, we get 
\begin{align*}
\frac{1}{2} \big(b_1^{(2)}(R_-)+b_1^{(2)}(S) \big)&= \dimNG{G} \ker \big(H_1^{(2)}(\partial M_-)\rightarrow H_1^{(2)}(M_-)\big)
\end{align*}By the standard inequality \[\dimNG{G}\ker (i\colon A\oplus B\to C)\leq \dimNG{G}\ker (i\colon 	A\to C)+\dimNG{G} B\] applied to $H^{(2)}_1(\partial M_-)= H^{(2)}_1(R_-)\oplus H^{(2)}_1(S)\to\HH_1(M_-)$ we obtain further
\begin{align*}
\frac{1}{2} \big(b_1^{(2)}(R_-)+b_1^{(2)}(S) \big)&= \dimNG{G} \ker \big(H_1^{(2)}(\partial M_-)\rightarrow H_1^{(2)}(M_-)\big) \\&\leq \dimNG{G} \ker\big(H_1^{(2)}(R_-)\rightarrow H_1^{(2)}(M_-)\big)+b_1^{(2)}(S) \\
	&\leq \dimNG{G} \ker\big(H_1^{(2)}(R_-)\rightarrow H_1^{(2)}(M)\big)+b_1^{(2)}(S) \\
	&\leq \dimNG{G} \image\big(H_2^{(2)}(M,R_-)\rightarrow H_1^{(2)}(R_-)\big)+b_1^{(2)}(S) \\
	&\leq b_2^{(2)}(M,R_-)+b_1^{(2)}(S). 
	\end{align*}
Recall that we dropped ``$\subset M$'' from the notation but by assumption $R_-$ and $S$ are incompressible and hence $\BB_1(R_-\subset M)=\BB_1(R_-)$ and $\BB_1(S\subset M)=\BB_1(S)$.
For surfaces with infinite fundamental group one has $-\chi(S)=b_1^{(2)}(S)$ and $-\chi(R_-)=b_1^{(2)}(R_-)$, which finishes the proof.
\end{proof}
Now the direction $(2)\Rightarrow (1)$ of the main theorem follows easily.
\begin{cor}
 Let $(M,\gamma)$ be an irreducible connected balanced sutured manifold with infinite fundamental group and such that $\gamma$ and $R_-$ are incompressible. If $\BB_2(M,R_-)=0$ then $M$ is taut. 
\end{cor}
\begin{proof}
Let $S$ be a surface obtained from Lemma~\ref{lem:sepsurf}. By construction of $S$ one has $-\chi(S)=x_M([R_-])$ and hence $\chi(S)-\chi(R_-)\geq 0$. By assumption we have $\BB_2(M,R_-)=0$ and Lemma~\ref{lem:inequaility} now implies that 
\[0\leq\chi(S)-\chi(R_-)=-x_M([R_-])-\chi(R_-)\leq 0.\] Therefore we get $x_M([R_-])=-\chi(R_-)$ and since $(M,\gamma)$ was balanced we obtain $x_M([R_-])=-\chi(R_+)$, too. But this is the definition of a taut sutured manifold.
\end{proof}
\subsection{Taut implies $\ell^2$-acyclic}
We are now going to show that if $(M,\gamma)$ is taut then $\HH_\ast(M,R_-)$ is zero.
Since $\ell^2$-Betti numbers are multiplicative under finite covers (Proposition \ref{prop:multiplicativ}) and a finite cover of a taut sutured manifold is again taut (Proposition~\ref{prop: tautfinitecover}) the above statement is true if and only if it is true for a finite cover. The proof consists of three steps. 
\begin{enumerate}
	\item There exists a finite cover $\widehat{M}\rightarrow M$ and a decomposition surface $S\subset\widehat{M}$ such that $\widehat{M}\sutdec{S} \widehat{M}'$ is a product sutured manifold.
	\item Denote by $\phi\in H^1(\widehat{M};\Z)$ the \Poincare dual of $S$, then $H_\ast^{\phi,(2)}(\widehat{M},\widehat{R_-})=0.$
	\item We use the second step and Schick's approximation result to conclude $\HH(M,R_-)=0$.
\end{enumerate}

For the first step we need the virtual fibering theorem due to Agol. Notice that the fundamental group of an irreducible 3-manifold with non-empty boundary is virtually RFRS ~\cite[Corollary
4.8.7]{AFW15}, which was proved by Przytycki and Wise~\cite{PW17}. We don't need the precise definition of RFRS, we only need that the following theorem holds in our situation.
\begin{theorem}\cite[Theorem 5.1]{Ag08}\label{thm:agol}
Let $M$ be an irreducible 3-manifold with empty or toroidal boundary. Assume $\pi_1(M)$ is infinite and virtually RFRS. If $\Sigma$ is a Thurston norm minimizing surface, then there is a finite sheeted cover $p\colon N\to M$ such that $p^{-1}(\Sigma)$ is a quasi-fiber. 
\end{theorem}
The next proposition is implicit in the article of Agol~\cite{Ag08}. We will show how it follows from Theorem ~\ref{thm:agol}.
\begin{prop}\label{prop:virtualvanishing}
	Let $(M,\gamma)$ be a connected taut sutured manifold with infinite fundamental group and such that $\gamma$ is incompressible. Then there exists a connected finite cover $(\widehat{M},\widehat{\gamma})$ and a decomposition surface $S$ in $\widehat{M}$, such that $\widehat{M}\sutdec{S}\widehat{M}'$ is a product sutured manifold and $[S]\in H_2(\widehat{M},\partial \widehat{M})$ is non-trivial.
\end{prop}
\begin{proof}
If $(M,\gamma)$ is a product sutured manifold then one can take the annulus obtained from a homologically non trivial curve in $R_-$ times the interval. Therefore we will in the following assume that $(M,\gamma)$ is a taut sutured manifold which is not a product. Recall the construction of the double $DM(\gamma)$ of a sutured manifold $M$:
\[ DM(\gamma):= M \sqcup_{R_\pm} M. \]
This is a 3-manifold, which is closed or has toroidal boundary. By our assumptions is $\gamma$ incompressible and $(M,\gamma)$ is taut. Hence $R_-\cup R_+$ is Thurston norm minimizing in $DM(\gamma)$ (\cite[Lemma 3.7]{Ga83}).
 
We can use the virtual fibering Theorem~\ref{thm:agol} to obtain a finite cover $p\colon W\to DM(\gamma)$ such that the surface $p^{-1}(R_- \cup R_+)$ is a quasi-fiber. We write $\Sigma=p^{-1}(R_- \cup R_+)$.
The taut sutured manifold $W'$ given by $W\sutdec{\Sigma} W'$ is by construction a finite cover of $M$. Therefore by Lemma~\ref{lem:quasifiber} applied to $\Sigma$ and $W$ we see that a connected component $\widehat{M}\subset W'$ is the desired finite cover of $M$.
\end{proof}
Now we can finally prove the rest of our main result.
\begin{theorem}\label{thm:lastthm}
	If $M$ is a connected taut sutured manifold with infinite fundamental group and $\gamma$ is incompressible, then $\BB_\ast (M,R_-)=0$.
\end{theorem}
\begin{proof}
Let $M$ be a taut sutured manifold. Recall that for any finite cover $\widehat{M}\to M$ we have
	\[[\widehat{M}:M]\cdot \BB_\ast(M,R_-) = \BB_\ast(\widehat{M},\widehat{R}_- ) \]
	and if $M$ is taut then $\widehat{M}$ is taut as well. Therefore it is sufficient to prove the theorem for a suitable finite cover. By this observation together with Proposition~\ref{prop:virtualvanishing} we will assume that $(M,\gamma)$ admits a decomposition surface $S$ such that $M'$ defined by $M\sutdec{S} M'$ is a product sutured manifold and $[S]\in H_2(M,\partial M)$ is non trivial. If we denote by $\phi\in H^{1}(M;\Z)$ the \Poincare dual of $S$ then by Lemma~\ref{lem:vanishing} we have 
	\[H_1^{\phi,(2)}(M,R_-)=0. \]   
Since the fundamental group of $M$ is residually finite we obtain a nested sequence $\pi_1(M)=\pi\supset\pi_1\supset \pi_2\ldots$ of normal subgroups such that $\pi/\pi_i$ is finite and $\bigcap_{i\in I}\pi_i=\left\{e\right\}$. Denote by $p_i\colon M_i\to M$ the corresponding finite cover.
Denote by $S_i:=p_i^{-1}(S)$ the pre-image of the surface $S$ and by $\phi_i=p_i^*(\phi)$ the pull back of $\phi$. Obviously, $S_i$ is the \Poincare dual of $\phi_i$. Furthermore, $M_i\sutdec{S_i} M_i'$ results in a product sutured manifold since $M_i'$ is a cover of $M'$ and $M'$ is a product sutured manifold. Hence we can apply Lemma~\ref{lem:vanishing} again to obtain \[H_1^{\phi_i,(2)}(M_i, {R_i}_-)=0.\]
Denote by $G_i=\ker (\phi) \cap \pi_i$ the kernel of $\phi_i$. We see that $G_i$ is normal in $\pi$. Moreover, by the third isomorphism Theorem we have 
	\[ (\pi /G_i)/(\pi_i/G_i)\cong \pi/\pi_i  \]
	and hence $(\pi_i/G_i)\triangleleft (\pi /G_i)$ is finite index. This yields
	\begin{align*}
	0=\dimNG{\Z}H_1^{\phi_i}(M_i, R_{i})&= \BB_1(\widetilde{M}/G_i,\widetilde{R}_-/G_i;\mathcal{N}(\pi_i/G_i))\\
	&=[\pi:\pi_i]\cdot\BB_1(\widetilde{M}/G_i, \widetilde{R}_-/G_i;\mathcal{N}(\pi/G_i)) 
	\end{align*}
	and in particular
	\[\BB_1(\widetilde{M}/G_i, \widetilde{R}_-/G_i;\mathcal{N}(\pi/G_i))=0. \]
	The groups $\pi/G_i$ are by construction virtually cyclic and hence lie in the class $\mathcal{G}$. We can now apply Schick's approximation Theorem \ref{thm:appschick} to the nested and cofinal sequence of normal subgroups $\pi_1(M )\supset G_1\supset G_2\ldots$ and obtain
	\[	\BB_1(M,R_-)=\lim_{i\to\infty} \BB_1 (\widetilde{M}/G_i,\widetilde{R_-} ;\mathcal{N}(\pi/G_i)) = 0 . \]
\end{proof}
\section{Applications}\label{sec:app} \noindent
In this section we prove Theorem~\ref{thm:haupt:luckfried}:
\begin{mytheorem}
	Let $N$ be a connected compact oriented irreducible 3-manifold with empty or toroidal boundary and let $\phi\in H^1(N;\Z)$ be a primitive cohomology class. We write $N_{\ker \phi}\to N$ for the cyclic covering corresponding to $\ker \phi$. We have
\[ b^{(2)}_1 (N_{\ker\phi}) = x_N( \phi ). \]
\end{mytheorem}
Note that $N_{\ker\phi}$ is not a finite CW-complex. We refer to Chapter 6 of L\"uck's monograph \cite{Lu02} for a definition of $\ell^2$-betti numbers  in this context. In the proof we are only using that the von Neumann dimension behaves well with respect to colimits of modules.

	\begin{proof}[Proof of Theorem~\ref{thm:haupt:luckfried}]
	Let $G=\ker \phi$ and $\Sigma$ be a Thurston norm minimizing decomposition surface \Poincare dual to $\phi$. We write $M= N\setminus \nu (\Sigma)$ and construct inductively
	\begin{align*}
	X_0&= M, \\
	X_n&= M\bigsqcup_{\Sigma_-=\Sigma_+} X_{n-1} \bigsqcup_{\Sigma_-=\Sigma_+}  M.
	\end{align*}
	By abuse of notation we write $\Sigma$ for $\Sigma_-$ in $X_0$. Since $\phi$ is primitive we have $\lim_{n\in\N} X_n = N_{\ker\phi}$. We refer to Figure~\ref{fig:coverequalsthurston} for an illustration of the situation.
	\begin{figure}
		\begin{center}
			\includegraphics[scale=1.5]{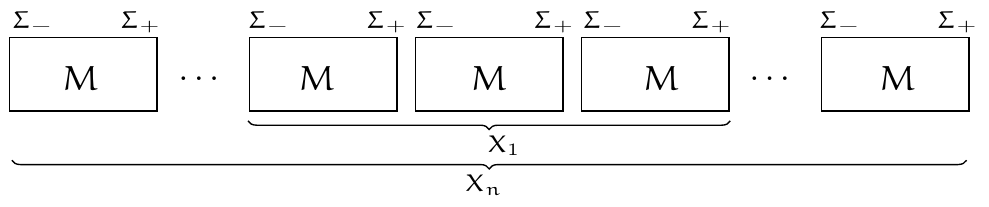}
			\caption{An illustration of the limiting process in the proof of Theorem~\ref{thm:haupt:luckfried}.}\label{fig:coverequalsthurston}
		\end{center}
	\end{figure} 
	
	Also note that all inclusions $\Sigma\to X_1\to X_n$ and $X_n\to N_\phi$ are $\pi_1$-injective and we can use the induction principle stated in Equation~(\ref{eq:induction}).
	
	By the excision isomorphism we have \[b^{(2)}(X_n,X_{n-1};\mathcal{N}(G))=b^{(2)}((M,\Sigma_+) \sqcup (M,\Sigma_-);\mathcal{N}(G))\] and hence by the main theorem
	\[ b^{(2)}(X_n,X_{n-1}; \mathcal{N}(G) )= b^{(2)}(X_n,X_{n-1})=0.\]
	We can consider the triple $(X_n,X_{n-1},\Sigma)$ and its associated long exact sequence in homology. It follows inductively that $b^{(2)}(X_n,\Sigma; \mathcal{N}(G))= 0$.  
	
	We also have the isomorphisms (see the PhD-thesis for more details \cite{He19}):
	
	\[\varinjlim_{n\in\N} \HH_*(X_n,\Sigma; \mathcal{N}(G))\cong\HH_*(\varinjlim_{n\in\N} X_n,\Sigma;\mathcal{N}(G)) =\HH_*(N_{\ker\phi},\Sigma;\mathcal{N}(G)).\]
	Then by cofinality of the von Neumann dimension (\cite[Theorem 6.13]{Lu02}) we have
	\[ b^{(2)}_*(N_{\ker\phi},\Sigma)= \lim_{n\to\infty} b^{(2)}_*(X_n,\Sigma; \mathcal{N}(G) ) = 0.\]
	We look at the long exact sequence in homology associated to the pair $(N_{\ker \phi}, \Sigma)$ and conclude from the additivity of the von Neumann dimension that \[ b^{(2)}_1(N_{\ker\phi})=b^{(2)}_1(\Sigma)=-\chi(\Sigma)= x_N(\phi) .\qedhere \]
\end{proof}

\end{document}